\documentclass{imanumao}
\usepackage{graphicx}
\jno{drnxxx}
\revised{}

\usepackage{amssymb}
\usepackage{amsmath}
\usepackage{amsthm}

\newtheoremstyle{assumption}{6pt}{6pt}{\rm}{}{\sffamily}{ }{ }{}
\theoremstyle{assumption}
\newtheorem{assumption}[theorem]{\sc Assumption}

\newtheoremstyle{remarkp}{6pt}{6pt}{\rm}{}{\sffamily}{ }{ }{}
\theoremstyle{remarkp}
\newtheorem{remarkp}[theorem]{\sc Remark}

\newtheoremstyle{lemmap}{6pt}{6pt}{\rm}{}{\sffamily}{ }{ }{}
\theoremstyle{lemmap}
\newtheorem{lemmap}[theorem]{\sc Lemma}

\newtheoremstyle{corollaryp}{6pt}{6pt}{\rm}{}{\sffamily}{ }{ }{}
\theoremstyle{corollaryp}
\newtheorem{corollaryp}[theorem]{\sc Corollary}

\usepackage{color}

\newcommand{\DEL}[1]{}

\newcommand{\Bs}{B-series}
\newcommand{\BsEx}[3]{B_{#1}(#2)(#3)}
\newcommand{\BsLaw}[2]{\widehat B_{#1}(#2)}
\newcommand{\BsLawIn}[2]{\widehat B_{#1}^{(i)}(#2)}

\newcommand{\MvI}[2]{I_{#1}(#2)}

\newcommand{\MvQ}[1]{\^I(#1)}
\newcommand{\MvQi}[2]{\^ I_{#1}(#2)}

\newcommand{\Dif}[1]{D^{#1}}

\newcommand{\NodeToBush}[2]{#2^{(#1)}}

\newcommand{\bush}[1]{[\,\treeroot\,]^{#1}}

\renewcommand{\^}{\widehat }

\newcommand{\abs}[1]{\left\vert #1 \right\vert}
\newcommand{\norm}[1]{\left\Vert #1 \right\Vert}

\newcommand{\bfr}{\mathbf{r}}
\newcommand{\bfu}{\mathbf{u}}
\newcommand{\calD}{{\mathcal D}}
\newcommand{\calF}{{\mathcal F}}
\newcommand{\calO}{{\mathcal O}}
\newcommand{\calT}{{\mathcal T}}
\newcommand{\C}{\mathbb C}
\newcommand{\N}{\mathbb N}
\newcommand{\R}{\mathbb R}
\newcommand{\Z}{\mathbb Z}
\newcommand{\Fe}{\Psi}
\newcommand{\Ge}{G}
\newcommand{\Gn}{\widehat G}
\newcommand{\Hper}{H_{\text per}}
\newcommand{\ii}{\hbox{\rm i}}
\newcommand{\IE}{i.~e.}

\newcommand{\ee}{\hbox{\rm e}}
\newcommand{\dd}{\hbox{\rm d}}

\newcommand{\stages}{s}
\newcommand{\tree}{\tau}
\newcommand{\sw}{h}

\newcommand{\comm}[1]{[ #1 ]}
\newcommand{\famF}{\mathcal F}
\newcommand{\CE}{C_\famF}

\newcommand{\monomz}{\mathbf z}

\usepackage{tikz,pgfplots}
\usetikzlibrary{arrows}

\newcommand{\treeroot}{
   \begin{tikzpicture}
   \filldraw [white](0,0) circle (0pt);
   \filldraw (0,0.1) circle (1.2pt);
   \end{tikzpicture}}
\newcommand{\treeba}[1]{
   \begin{tikzpicture}[scale=#1]
      \draw[line width=0.4] (0,0) -- (0.15,0.15);
      \filldraw (0,0) circle (1.2pt);
      \filldraw (0.15,0.15) circle (1.2pt);
   \end{tikzpicture}}

\usepackage{soul,xcolor}
\setstcolor{red}

\begin{document}

\title{On the convergence of Lawson methods for semilinear stiff problems}%


\author{{\sc Marlis Hochbruck\thanks{marlis.hochbruck@kit.edu}}\\[2pt]
Institut f\"ur Angewandte und Numerische Mathematik,
Karlsruher Institut f\"ur Technologie,\\
D-76149 Karlsruhe, Germany\\[2pt]
{\sc and\\[2pt]
Alexander Ostermann\thanks{alexander.ostermann@uibk.ac.at}}\\[2pt]
Institut f\"ur Mathematik, Universit\"at Innsbruck, A-6020 Innsbruck, Austria
}


\maketitle

\begin{abstract}
{Since their introduction in 1967, Lawson methods have achieved constant interest in the time discretization of evolution equations. The methods were originally devised for the numerical solution of stiff differential equations. Meanwhile, they constitute a well-established class of exponential integrators. The popularity of Lawson methods is in some contrast to the fact that they may have a bad convergence behaviour, since they do not satisfy any of the stiff order conditions. The aim of this paper is to explain this discrepancy. It is shown that non-stiff order conditions together with appropriate regularity assumptions imply high-order convergence of Lawson methods. Note, however, that the term regularity here includes the behaviour of the solution at the boundary. For instance, Lawson methods will behave well in the case of periodic boundary conditions, but they will show a dramatic order reduction for, e.g., Dirichlet boundary conditions. The precise regularity assumptions required for high-order convergence are worked out in this paper and related to the corresponding assumptions for splitting schemes. In contrast to previous work, the analysis is based on expansions of the exact and the numerical solution along the flow of the homogeneous problem. Numerical examples for the Schr\"{o}dinger equation are included.}
{exponential integrators; Lawson methods; linear and nonlinear Schr\"{o}dinger equations; evolution equations; order conditions.}
\end{abstract}

\section{Introduction} \label{sec:intro}

Exponential integrators are a well-established class of methods for the numerical solution of semilinear stiff differential equations. If the stiff initial value problem stems from a spatial semi-discretization of an evolutionary partial differential equation (PDE), the very form of the domain of the spatial differential operator enters the convergence analysis. The \emph{stiff order conditions}, which guarantee a certain order of convergence independently of the considered problem, must be independent of the domain of this operator (which, in general, involves certain boundary conditions). This is the main reason why stiff order conditions for exponential integrators are quite involved (see \cite{HocO05SIAM} and \cite{LuaO13}).

For particular problems, however, less conditions are required for obtaining a certain order of convergence. (The same is true for ordinary differential equations (ODEs), where linear problems, e.g., require less order conditions for Runge--Kutta methods than nonlinear ones.) It was already observed in \cite{HocO05Apnum} that periodic boundary conditions do not give any order reduction in exponential integrators of collocation type in contrast to homogeneous Dirichlet boundary conditions, which restrict the order of convergence considerably (close to the stage order, depending on the precise situation). Full-order convergence for periodic boundary conditions was also noticed in \cite{KasT05} and \cite{BesDLV17}.

A similar behaviour can be observed for Lawson methods which are obtained by a linear variable transformation from (explicit) Runge--Kutta methods (see \cite{Law67} and Section~\ref{sec:lawson} below). These methods are very attractive, since they can be easily constructed from any known Runge--Kutta method. Unfortunately, Lawson methods exhibit a strong order reduction, in general. For particular problems, however, they show full order of convergence (see \cite{CanG15P}, \cite{BalFMMTP16}, and \cite{MonB16}). By construction, Lawson methods do satisfy the order conditions for non-stiff problems. Such conditions will be called \emph{non-stiff} or \emph{conventional order conditions} henceforth. However, Lawson methods \emph{do not} satisfy any of the stiff order conditions, as detailed in \cite{HocO05SIAM}, \cite{HocO10}, and \cite{LuaO13}. This fact can result in a dramatic order reduction, even down to order one for parabolic problems with homogeneous Dirichlet boundary conditions.

So far, the derivation of (stiff) order conditions for exponential integrators was based on standard expansions of the exact and the numerical solution. There, the main assumption on the problem is that the exact solution and its composition with the nonlinearity are both sufficiently smooth in time (see \cite{HocO05SIAM} and \cite{LuaO13}). Any additional regularity in space is not of immediate benefit in this analysis. This is in contrast to splitting methods, where spatial regularity usually shows up in form of commutator bounds (see, e.g., \cite{JahL00}).

In this paper, we study the convergence behaviour of Lawson methods for semilinear problems. One of the main contributions of this paper is a different expansion of the solution. It is still based on the variation-of-constants formula but the nonlinearity is expanded along the flow of the homogeneous problem. This expansion can be derived in a systematic way using trees as in \cite{HaiNW93} and \cite{LuaO13}. The expansion of the exact solution is carried out in terms of elementary integrals, that of the numerical solution in terms of elementary quadrature rules. We show that conventional, non-stiff order conditions together with (problem-dependent) assumptions on the exact solution give full order of convergence. This involves regularity of the solution in \emph{space and time}. Our main result for Lawson methods is stated in Theorems~\ref{thm:lawson-order} and \ref{thm:lawson-error}. We prove that Lawson methods converge with order $p$, if the order of the underlying Runge--Kutta methods is at least $p$ and the solution satisfies appropriate regularity assumptions. These conditions are studied in detail for methods of orders one and two, respectively, and they are related to the corresponding conditions that arise in the analysis of splitting methods. In particular, this is worked out for the nonlinear Schr\"{o}dinger equation. Our error analysis also reveals a different behaviour between the first-order Lawson method and the exponential Euler method, which is visible in numerical experiments.

The outline of the paper is as follows. In Section~\ref{sec:lawson}, we recall the construction of Lawson methods. The expansion of the numerical and the exact solution in terms of elementary integrals is given in Section~\ref{sec:expansion}. There, we also introduce the analytic (finite dimensional) framework which typically occurs when discretizing a semilinear parabolic or hyperbolic PDE in space. Order conditions and convergence results are given in Section~\ref{sec:trees}. The resulting regularity assumptions are discussed in Section~\ref{sec:applications}. These assumptions are related to the corresponding conditions for splitting methods. Numerical examples that illustrate the required regularity assumptions and the proven convergence behaviour are also presented.

\section{Lawson methods}\label{sec:lawson}
Consider a semilinear system of stiff differential equations
\begin{equation}  \label{eq:semi-na}
  u'(t) + A u(t) = g\bigl(t, u(t) \bigr), \qquad u(0) = u_0,
\end{equation}
where the stiffness stems from the linear part of the equation, i.e., from $A$, which is either an unbounded linear operator or its spatial discretization, i.e., a matrix. The precise assumptions on $A$ and $g$ will be given in Section~\ref{sec:expansion}. For the numerical solution of \eqref{eq:semi-na}, \cite{Law67} considered the following change of variables:
\begin{equation*}
  w(t) = \ee^{tA} u(t).
\end{equation*}
Note that when applied to evolution equations, this transformation has to be done in a formal way, since $\ee^{t A}$ might not be a meaningful object in our general framework.

Inserting the new variables into \eqref{eq:semi-na} gives the transformed differential equation
\begin{equation}\label{eq:semi-w}
  \begin{aligned}
    w'(t) &= \ee^{tA} \bigl( u'(t)+ A u(t)\bigr)\\
    & = \ee^{tA} g\bigl(t,\ee^{-tA} w(t)\bigr),\qquad w(0)=u_0.
  \end{aligned}
\end{equation}
For the solution of this problem, an $\stages$-stage explicit Runge--Kutta method with coefficients $b_i,c_i, a_{ij}$ is considered. The method is assumed to satisfy the simplifying assumptions $c_1=0$ and
\begin{equation}\label{eq:simpa}
\sum_{j=1}^{i-1} a_{ij} = c_i,\qquad i=2,\ldots,\stages.
\end{equation}
Transforming the Runge--Kutta discretization of \eqref{eq:semi-w} back to the original variables yields the corresponding Lawson method for \eqref{eq:semi-na}
\begin{subequations}  \label{eq:lawson}
\begin{align}
  u_{n+1} & = \ee^{-\sw A} u_n + \sw \sum_{i=1}^\stages b_i \ee^{-(1-c_i)\sw A} G_{ni}, \label{eq:lawson-u1}\\
   G_{ni} & = g\bigl(t_n+c_i\sw, U_{ni}), \label{eq:lawson-Gni}\\
  U_{ni} & = \ee^{-c_i \sw A} u_n + \sw \sum_{j=1}^{i-1} a_{ij} \ee^{-(c_i-c_j)\sw A} G_{nj}, \qquad i=1,\ldots,\stages.   \label{eq:lawson-Ui}
\end{align}
\end{subequations}
Here, $u_n$ is the numerical approximation to the exact solution $u(t)$ at time $t=t_n=n \sw$, and $\sw$ is the step size. Note that this method makes explicit use of the action of the matrix exponential function. Depending on the properties of $A$, the nodes $c_1,\ldots,c_s$ have to fulfill  particular assumptions, see Assumption~\ref{ass:noframe} in the next section. Because of these actions of the matrix exponential, Lawson methods form a particular class of exponential integrators. For a review on such integrators, we refer to \cite{HocO10}.

For a non-stiff ordinary differential equation \eqref{eq:semi-na}, it is obvious that the order of the Runge--Kutta method applied to \eqref{eq:semi-w} coincides with that of the corresponding Lawson method applied to \eqref{eq:semi-na}. It is the aim of this paper to show that this is also true in the stiff situation, if appropriate
regularity assumptions hold (we will explain the meaning of regularity in the context of discretized PDEs in Section~\ref{sec:applications}).

\section{Expansion of the exact and the numerical solution} \label{sec:expansion}

By adding $t'=1$ to \eqref{eq:semi-na}, the differential equation is transformed to autonomous form. It is well known that Runge--Kutta methods of order at least one satisfying~\eqref{eq:simpa} are invariant under this transformation. Therefore, we restrict ourselves henceforth to the autonomous problem
\begin{equation}  \label{eq:semi-Lawson}
  u'(t) + A u(t) = g\bigl(u(t) \bigr), \qquad
    u(0) = u_0.
\end{equation}
Let $X$ be a Hilbert space or a Banach space with norm $\| \cdot \|$. Our main assumptions on $A$ and $g$ are as follows.
\begin{assumption}\label{ass:noframe}
  Let $A$  belong to a family $\mathcal F$ of linear operators on $X$ such that $-A$
  generates a group satisfying
  \begin{equation}\label{eq:mat-bound}
    \left\| \ee^{-tA}\right\| \le \CE,
  \end{equation}
  with a moderate constant $\CE$, uniformly for all $t \in \R$ and all operators $A\in \mathcal F$. It is
  sufficient to require that $-A$ generates a bounded semigroup (\IE,
  \eqref{eq:mat-bound} for $t\ge 0$),  if the nodes
  $c_i$ of the considered explicit Runge--Kutta method are ordered as
  $0=c_1\le c_2\le \ldots \le c_\stages\le 1$.
\end{assumption}

The set of infinitesimal generators of non-expansive (semi)groups in $X$ is a possible choice for the family $\mathcal F$. In addition, the above assumption is typically satisfied in situations where \eqref{eq:semi-Lawson} stems from a spatial discretization of a semilinear parabolic or hyperbolic partial differential equation. The important fact here is that the constant $\CE$ is independent of the spatial mesh width for finite difference and finite element methods, and independent of the number of ansatz functions in spectral methods. As our error bounds derived below do not depend on $A$ itself but only on the constant $\CE$, they also apply to spatially discretized systems.

\begin{assumption}\label{ass:noframeg}
  For a given integer $p\geq 0$, the nonlinearity $g$ is
  $p$ times  differentiable with bounded derivatives in a
  neighborhood of the solution of \eqref{eq:semi-Lawson}.
\end{assumption}

We recall that the solution of~\eqref{eq:semi-Lawson} can be represented in terms of the variation-of-constants formula
\begin{equation*}   
  u(\theta\sw) = \ee^{-\theta\sw A} u_0
   + \sw \int_0^\theta \ee^{-(\theta-\sigma)\sw A} g\bigl(u(\sigma\sw)\bigr) \dd \sigma.
\end{equation*}
Applying this formula recursively and expanding the nonlinearity along the flow of the homogeneous problem yields the following expansion of the exact solution
\begin{equation}\label{eq:exp-ex}
\begin{aligned}
  u(\sw) & =  \ee^{-\sw A} u_0 + \sw \int_0^1  \ee^{-(1 -\sigma) \sw A}
  g\Bigl( \ee^{-\sigma \sw A} u_0 + \sw \int_0^\sigma \ee^{-(\sigma-\eta) \sw A}
  g\bigl(u(\eta\sw)\bigr) \dd \eta \Bigl) \dd \sigma \\
     &=  \ee^{-\sw A} u_0 + \sw  \int_0^1 \ee^{-(1 -\sigma) \sw A}
       g_\sigma  \dd \sigma \\
     & \qquad +  \sw^2 \int_0^1  \ee^{-(1 -\sigma)\sw A}g'_\sigma
      \int_0^\sigma \ee^{-(\sigma-\eta) \sw A}  g_\eta \dd\eta\dd\sigma\\
     & \qquad +  \sw^3 \int_0^1 \ee^{-(1 -\sigma)\sw A}g'_\sigma
      \int_0^\sigma \ee^{-(\sigma-\eta) \sw A}  g'_\eta
      \int_0^\eta \ee^{-(\eta-\xi) \sw A}  g_\xi \dd\xi \dd\eta\dd\sigma\\
    & \qquad + \tfrac12 \sw^3  \int_0^1  \ee^{-(1 -\sigma) \sw A} g_\sigma''
      \Bigl( \int_0^\sigma \ee^{-(\sigma-\eta)\sw A}  g_\eta  \dd \eta,
       \int_0^\sigma \ee^{-(\sigma-\xi)\sw  A}  g_\xi \dd \xi
 \Bigr)\dd \sigma + \calO(\sw ^4),
\end{aligned}
\end{equation}
where we have used the shorthand notation
\begin{equation}\label{eq:gketa}
  g_\eta = g_\eta(u_0)= g\bigl(\ee^{-\eta \sw A} u_0 \bigr), \qquad
  g_\eta^{(k)} = g_\eta^{(k)}(u_0) = g^{(k)}\bigl(\ee^{-\eta \sw A} u_0 \bigr), \qquad k\geq 1.
\end{equation}
Note that here and throughout the whole section, the constant in the Landau symbol $\calO$ only depends on $\CE$ and the derivatives of $g$, but not explicitly on $A$ itself, i.e., not on the stiffness. Also note that this expansion differs considerably from the previous work (see, e.g., \cite{HocO05SIAM, LuaO13}) where the nonlinearity $g(u(t))$ was expanded with respect to $t$.

Next we perform a similar expansion of the numerical solution \eqref{eq:lawson}, which yields (again in the autonomous case)
\begin{equation}\label{eq:exp-num}
\begin{aligned}
  u_1 & = \ee^{-\sw A} u_0
    + \sw \sum_{i=1}^\stages b_i \ee^{-(1-c_i)\sw A} g_{c_i }\\
   & \qquad + \sw^2 \sum_{i=1}^\stages b_i \ee^{-(1-c_i)\sw A}
     g'_{c_i } \sum_{j=1}^{i-1} a_{ij}
                 \ee^{-(c_i-c_j)\sw A} g_{c_j }\\
   & \qquad + \sw^3 \sum_{i=1}^\stages b_i \ee^{-(1-c_i)\sw A}
     g'_{c_i } \sum_{j=1}^{i-1} a_{ij}
                 \ee^{-(c_i-c_j)\sw A} g'_{c_j }
                \sum_{k=1}^{j-1} a_{jk}
                 \ee^{-(c_j-c_k)\sw A} g_{c_k }\\
   & \qquad + \tfrac12 \sw^3 \sum_{i=1}^\stages b_i \ee^{-(1-c_i)\sw A}
     g''_{c_i } \Bigl( \sum_{j=1}^{i-1} a_{ij}
                 \ee^{-(c_i-c_j)\sw A} g_{c_j },
                \sum_{k=1}^{i-1} a_{ik}
                 \ee^{-(c_i-c_k)\sw A} g_{c_k }  \Bigr) + \calO(\sw^4).
\end{aligned}
\end{equation}
As we have used the variation-of-constants formula and its discrete counterpart, respectively, the expansions of the exact and the numerical solution reflect the well-known tree structure of (explicit) Runge--Kutta methods.
In the following we use the classic trees which are well-established for studying the non-stiff order conditions for Runge--Kutta methods, see \cite[Section~II.2]{HaiNW93}, \cite[Section~III.1]{HaiLW06}, and references given there.

By $\calT$ we denote the set of unlabeled rooted trees. We recall that these trees are defined recursively by
\begin{enumerate}
\item $\,\treeroot\, \in \calT$,
\item if $\tree_1,\ldots,\tree_k \in \calT$, then $[\tree_1,\ldots,\tree_k] \in \calT$.
\end{enumerate}
Here, $[\tree_1,\ldots,\tree_k]$ (a $k$ tuple without ordering) denotes the
tree which is obtained by concatenating the roots of the trees
$\tree_1,\ldots,\tree_k$ via $k$ branches with a new node. This node becomes
the root of the tree [$\tree_1,\ldots,\tree_k$].

For $\tree \in \calT$ the elementary differential $D(\tree)$ of a
smooth function $g$ is defined recursively in the following way. For $\tree = \treeroot$ we
have $D(\treeroot)(w) = g\bigl(w\bigr)$, and for $\tree=
[\tree_1,\ldots,\tree_k]$ we have
    \begin{equation*}
      D(\tree)(w) =
            g^{(k)}\bigl( w \bigr)
            \bigl( D(\tree_1)(w),\ldots,D(\tree_k)(w)\bigr).
    \end{equation*}
By $\varrho(\tree)$ we denote the order of the tree $\tree$ which is defined as the number of nodes of $\tree \in \calT$. The trees of order less or equal then $p$ are denoted by
\begin{equation*}
  \calT_p = \{ \tree \in \calT \mid \varrho(\tree) \le p \} .
\end{equation*}

Motivated by the expansion \eqref{eq:exp-ex} of the exact solution we define elementary integrals.
\begin{definition}  \label{def:Ge}
  For $\tree \in \calT$ and $0 \leq \zeta \leq 1$ we define
  the \emph{elementary integral} $\Ge_\zeta(\tree)$, its \emph{integrand} $\Psi_\zeta(\tree)$ and the \emph{multivariate integration operator} $\MvI{\zeta}{\tree}$ recursively
  in the following way.
  \begin{enumerate}
  \item For $\tree = \,\treeroot\,$ and a univariate function $f$, we set
    \begin{align*}
      \Psi_\zeta(\treeroot)(\sigma,w) &= \ee^{-(\zeta-\sigma) \sw A} g_\sigma\bigl( w\bigr),\\
      \MvI{\zeta}{\treeroot}f &= \int_0^\zeta f(\sigma)\dd\sigma.
    \end{align*}
  \item For $\tree= [\tree_1,\ldots,\tree_k]$ and a multivariate function $f$ in $\rho(\tree)$ variables, we set
    \begin{align*}
      \Psi_\zeta(\tree)(\sigma,\cdot_1,\ldots,\cdot_k,w) &= \ee^{-(\zeta-\sigma) \sw A}
            g_\sigma^{(k)}\bigl(w \bigr)
            \bigl(\Psi_\sigma(\tree_1)(\cdot_1,w),\ldots,\Psi_\sigma(\tree_k)(\cdot_k,w)\bigr), \\
      \MvI{\zeta}{\tree}f &= \int_0^\zeta \MvI{\sigma}{\tree_1}\cdots\MvI{\sigma}{\tree_k} f(\sigma,\cdot_1,\ldots,\cdot_k)\dd\sigma.
    \end{align*}
    Here, $\cdot_j$ refers to the variables corresponding to the
      $j$th subtree $\tree_j$.
  \end{enumerate}
Finally, we define for all $\tree\in\calT$ the \emph{elementary integrals} as
$$
\Ge_\zeta(\tree)(w) = \MvI{\zeta}{\tree} \Fe_\zeta(\tree) (\cdot,w)
$$
and set $\Psi(\tree) = \Psi_1(\tree)$, $\MvI{}{\tree}=\MvI{1}{\tree}$, and $\Ge(\tree) = \Ge_1(\tree)$.
\end{definition}

For example, we have $\treeba{1} = [\,\treeroot\,]$ and
\begin{equation*}
\Fe(\treeba{1}\,)(\sigma_1,\sigma_2,w) = \ee^{-(1-\sigma_1) \sw A}  g_{\sigma_1}'(w)\ee^{-(\sigma_1-\sigma_2) \sw A}  g_{\sigma_2}(w).
\end{equation*}

It is straightforward to verify that the elementary integrals satisfy the recurrence relation
\begin{align*}
      \Ge_\zeta(\treeroot)(w)  &= \int_0^\zeta \ee^{-(\zeta-\sigma) \sw A}
            g_\sigma\bigl(w\bigr) \dd \sigma,\\
      \Ge_\zeta(\tree)(w) &= \int_0^\zeta \ee^{-(\zeta-\sigma) \sw A}
            g_\sigma^{(k)}\bigl(w \bigr)
            \bigl(\Ge_\sigma(\tree_1)(w),\ldots,\Ge_\sigma(\tree_k)(w)\bigr) \dd \sigma
\end{align*}
for $\tree=[\tree_1,\ldots,\tree_k]$.

Our assumptions on  $g$ and $A$ ensure that the integrand $\Fe_\zeta(\tree)(\cdot,w)$
is bounded if $\tree \in \calT_{p+1}$ for $w$ in a neighborhood of the
exact solution of \eqref{eq:semi-Lawson} and $\sw$ sufficiently small.

\begin{remarkp}
  In the nonstiff situation, where $A\equiv 0$, all evaluations of $g$
  or its derivatives are at the fixed value $w$. Thus
  $\Ge_\zeta(\tree)(w)$ reduces to a multivariate integral over the
  constant integrand $\Fe_\zeta(\tree)(\cdot,w) \equiv D(\tree)(w)$.
\end{remarkp}

The following theorem shows how the expansion \eqref{eq:exp-ex} can
be expressed as a (truncated) \Bs. Here we use the notation from
\cite[Section~III.1]{HaiLW06}.


\begin{theorem}  \label{thm:uex}
  The exact solution of \eqref{eq:semi-na} satisfies
  \begin{equation}
    \label{eq:ex-expand-a}
    u(\zeta \sw) = \ee^{-\zeta \sw A} u_0 + \BsEx{p}{u_0}{\zeta}
        + \calO(\sw^{p+1}), \qquad \zeta \in [0,1],
  \end{equation}
  where  we define the \Bs \  for $w \in X$ and $\zeta \in [0,1]$ as
  \begin{align*}
    \BsEx{p}{w}{\zeta} =  \sum_{\tree \in \calT_p} \frac{\sw^{\varrho(\tree)}}{\sigma(\tree)} \Ge_\zeta(\tree)(w) ,
  \end{align*}
  with the symmetry coefficients $\sigma( \treeroot) = 1$ and $ \sigma([\tree_1,\ldots,\tree_k]) = \sigma(\tree_1)\cdots \sigma(\tree_k) \mu_1! \mu_2 ! \cdots$.  The integers $\mu_1,\mu_2, \ldots$ specify the number of equal trees among $\tree_1, \ldots, \tree_k$.
\end{theorem}

\begin{proof}
The proof is done by induction. For $p=0$, the claim follows from the variation-of-constants formula and $\BsEx{0}{u_0}{\zeta} = 0$, since
  \begin{equation}  \label{eq:uzetatau}
    u(\zeta \sw)  = \ee^{-\zeta \sw A} u_0 + \sw \int_0^\zeta  \ee^{-(\zeta -\sigma) \sw A} g\bigl(u(\sigma \sw)\bigr) \dd \sigma = \ee^{-\zeta \sw A} u_0 
    + \calO(\sw) .
  \end{equation}
The induction step follows the lines of the proof of \cite[Lemma~III.1.9]{HaiLW06} with the following modifications: we use the variation-of-constants formula and truncate the series in such a way that only the first $p$ derivatives of $g$ enter the expansion. We omit the details.
\end{proof}

Now we proceed analogously for the numerical solution starting with the definition of elementary quadrature rules.

\begin{definition}  \label{def:Gn}
For $\tree \in \calT$ we define the \emph{multivariate quadrature
  operators} $\MvQ{\tree}$  and $ \MvQi{i}{\tree}$, $i=1,\ldots,s$,
 recursively in the following way.
\begin{enumerate}
    \item For $\tree = \,\treeroot\,$ and a univariate function $f$, we set
      \begin{align*}
      \MvQ{\treeroot}f = \sum_{j=1}^s b_j f(c_j),\qquad
      \MvQi{i}{\treeroot}f = \sum_{j=1}^{i-1} a_{ij} f(c_j),\quad 1\le i \le s.
      \end{align*}
    \item For $\tree= [\tree_1,\ldots,\tree_k]$ and a multivariate function $f$ in $\rho(\tree)$ variables, we set
      \begin{align*}
      \MvQ{\tree}f &= \sum_{j=1}^s b_j \MvQi{j}{\tree_1}\cdots\MvQi{j}{\tree_k} f(c_j,\cdot_1,\ldots,\cdot_k),\\
      \MvQi{i}{\tree}f &= \sum_{j=1}^{i-1} a_{ij} \MvQi{j}{\tree_1}\cdots\MvQi{j}{\tree_k} f(c_j,\cdot_1,\ldots,\cdot_k),\quad 1\le i \le s.
      \end{align*}
    \end{enumerate}
Finally, we define the \emph{elementary quadrature rules} in the following way
    \begin{align*}
      \Gn(\tree)(w) = \MvQ{\tree} \Fe(\tree)(\cdot,w),\qquad
      \Gn_i(\tree)(w) = \MvQi{i}{\tree} \Fe_{c_i}(\tree)(\cdot,w), \quad  1\le i\le s.
    \end{align*}
\end{definition}

It is easy to see from the recursive definitions that the elementary quadrature rules satisfy
      \begin{align*}
        \Gn(\treeroot)(w)
        &= \sum_{j=1}^\stages b_j  \ee^{-(1-c_j)\sw A}
          g_{c_j}\bigl( w \bigr),\\
        \Gn_i(\treeroot)(w)
        &= \sum_{j=1}^{i-1} a_{ij}  \ee^{-(c_i-c_j)\sw A}
          g_{c_j}\bigl( w \bigr), \quad  1\le i\le s
      \end{align*}
and
      \begin{align*}
        \Gn(\tree)(w)
        &= \sum_{j=1}^\stages b_j  \ee^{-(1-c_j)\sw A}
          g_{c_j}^{(k)} \bigl( w \bigr)
          \bigl( \Gn_j(\tree_1)(w),\ldots,\Gn_j(\tree_k)(w)\bigr), \\
        \Gn_i(\tree)(w)
        &= \sum_{j=1}^{i-1} a_{ij}  \ee^{-(c_i-c_j)\sw A}
          g_{c_j}^{(k)} \bigl( w \bigr)
          \bigl( \Gn_j(\tree_1)(w),\ldots,\Gn_j(\tree_k)(w)\bigr), \quad  1\le i\le s
      \end{align*}
for $\tree= [\tree_1,\ldots,\tree_k]$. This allows us to express the expansion \eqref{eq:exp-num} of
the numerical solution in terms of elementary quadrature rules.

\begin{theorem}   \label{thm:lawson-expand}
  The numerical solution of \eqref{eq:semi-Lawson}  satisfies
  \begin{subequations}
  \begin{align}
    \label{eq:num-expand-a}
    u_{1} &= \ee^{- \sw A} u_0 + \BsLaw{p}{u_0}
        + \calO(\sw^{p+1}), \\
       \label{eq:num-expand-a-inner}
    U_{0i} &= \ee^{- c_i \sw A} u_0 + \BsLawIn{p}{u_0} +  \calO(\sw^{p+1}), \qquad i = 1,\ldots,s    ,
  \end{align}
  \end{subequations}
  where we define the numerical \Bs \  for $w \in X$  as
  \begin{align*}
    \BsLaw{p}{w} &=  \sum_{\tree \in \calT_p}
                   \frac{\sw^{\varrho(\tree)}}{\sigma(\tree)}
                   \Gn(\tree)(w), \\
                   \BsLawIn{p}{w} &=  \sum_{\tree \in \calT_p} \frac{\sw^{\varrho(\tree)}}{\sigma(\tree)} \Gn_i(\tree)(w),
                   \quad i=1,\ldots,s.
  \end{align*}
\end{theorem}
\begin{proof}
The proof is done analogously to the proof of Theorem~\ref{thm:lawson-expand}.
\end{proof}

\section{Order conditions and convergence} \label{sec:trees}

In this section we present a systematic way of deriving general stiff
convergence results for Lawson methods based on trees.

The expansions of the exact and the numerical solution
in terms of elementary integrals and elementary quadrature rules
derived in the previous section allow us to study the local error in the same way as for
classical Runge--Kutta methods.
In fact we show that the orders of these quadrature rules determine
the local error of the Lawson method. A similar strategy was used in
the analysis of splitting methods by \cite{JahL00}.  General stiff
order conditions for exponential Runge--Kutta methods have been
derived in \cite{LuaO13} and for splitting methods in \cite{HanO16}.

As usual, we say that the Lawson method is of (stiff) order $p$ if the local
error satisfies
\begin{equation*}
  \norm{u(\sw) - u_1} \leq C \sw^{p+1}
\end{equation*}
uniformly for smooth nonlinearities and operators $A$ satisfying Assumption~\ref{ass:noframe},
meaning that the constant $C$ depends on the constant $\CE$ defined in
\eqref{eq:mat-bound} but not  on $A$ itself.

\begin{theorem}  \label{thm:order}
  The Lawson method \eqref{eq:lawson} is of order $p$ if
  \begin{equation*}
    \Gn(\tree)(u_0) - \Ge(\tree)(u_0) = \calO(\sw^{p+1-\varrho(\tree)}),
    \qquad \text{for all } \tree \in \calT_p.
  \end{equation*}
\end{theorem}
\begin{proof}
  From Theorems~\ref{thm:uex} and \ref{thm:lawson-expand} we have
  \begin{equation} \label{eq:local-err-Bseries}
    u(\sw) - u_1 =
    \sum_{\tree \in \calT_p} \frac{\sw^{\varrho(\tree)}}{\sigma(\tree)}
    \big(\Ge(\tree)(u_0)-\Gn(\tree)(u_0) \big) + \calO(\sw^{p+1}).
  \end{equation}
  This proves the statement.
\end{proof}

\begin{remarkp}
  The above derivation can be easily generalized to exponential
  integrators with a fixed linearization
  \begin{align*}
    U_i & = \ee^{-c_i \sw A} u_0
      + \sw \sum_{j=1}^{i-1} a_{ij}(-\sw A) g(U_j),
     \qquad i=1,\ldots,\stages,  
    \\
    u_1 & = \ee^{-\sw A} u_0
      + \sw \sum_{i=1}^\stages b_i(-\sw A) g(U_i),
  \end{align*}
  cf.\ \cite{HocO10}.  If one
  replaces $b_i \ee^{-(1-c_i)\sw A}$ by $b_i(-\sw A)$ and
  $a_{ij} \ee^{-(c_i-c_j)\sw A}$ by $a_{ij}(-\sw A)$ in
  Definition~\ref{def:Gn}, Theorems~\ref{thm:lawson-expand} and
  \ref{thm:order} also hold for general exponential Runge--Kutta
  methods.  If the stiff order conditions derived in
  \cite{HocO05SIAM} and \cite{LuaO13}
   are satisfied up to order $p$, then
  $\Gn(\tree)(u_0) - \Ge(\tree)(u_0) = \calO(\sw^{p+1-\varrho(\tree)})$
  for all $ \tree \in \calT_p$.
\end{remarkp}

\begin{example}
  For the exponential Euler method, where $s=1$,
  $c_1=0$, and $b_1(z) = \varphi_1(z)$, we have
  \begin{equation*}
    \Gn(\treeroot)(u_0) = b_1(-\sw A) g(u_0) = \varphi_1(-\sw A) g(u_0)
    = \int_0^1 \ee^{-(1-\sigma)\sw A}  g(u_0) \dd\sigma
  \end{equation*}
  and thus
    \begin{align*}
      \Gn(\treeroot)(u_0) - \Ge(\treeroot)(u_0) &=
           \int_0^1 \ee^{-(1-\sigma)\sw A} \bigl( g(u_0) -  g_\sigma(u_0)\bigr)
                \dd\sigma.
    \end{align*}
    The condition for order one
    requires that  $\norm{g(u_0) - g_\sigma(u_0)} \leq C \sw$.
    In the linear case, where $g(u)=Bu$, this can be written as
    \begin{equation}  \label{eq:euler-cond}
      \sw^{-1} \bigl(g(u_0) - g_\sigma(u_0)\bigr)
         = \sw^{-1} B \bigl( I - \ee^{-\sigma \sw A} \bigr) u_0
         =  B \varphi_1(-\sigma \sw A) \sigma A u_0.
    \end{equation}
    Hence the condition is fulfilled if $Au_0$ is uniformly bounded, \IE,
    $u_0 \in \calD(A)$. For the convergence, we thus need
    $u(t) \in \calD(A)$ for $t\in [0,T]$.

    It might be interesting to compare \eqref{eq:euler-cond}
    to the condition given in
    \cite[Lemma~2.13]{HocO10}
    which was proved by a Taylor series
    expansion of $g\bigl(u(t)\bigr)$. For linear problems, it reads
    \begin{equation}
      \label{eq:euler-acta}
      \norm{ B (Au(t)+Bu(t))} \leq C.
    \end{equation}
    Hence both results require the same regularity, namely that $Au(t)$
    is uniformly bounded. Note, however, that \eqref{eq:euler-acta}
    does not involve the $\varphi_1$ function. The latter decays like
    $1/z$ as $z\to\infty$ in the closed left half-plane, hence
    components corresponding to eigenvalues with large negative real
    part are damped.
\end{example}

\begin{corollaryp}  \label{cor:rkm-qf}
  If the underlying Runge--Kutta method is of
  (conventional) order $p$ then
  \begin{equation} \label{eq:Gn-error}
    \MvI{}{\tree} 1 =  \MvQ{\tree} 1
    \qquad \text{for all } \tree \in \calT_p,
  \end{equation}
  where $1: [0,1]^{\varrho(\tree)} \to \R: \sigma \mapsto 1$ denotes the multivariate constant function with value one.
  %
\end{corollaryp}

\begin{proof}
  First note that for problems with $A\equiv 0$, we have
  \begin{equation*}
    \Ge(\tree)(w) = \MvI{}{\tree} D(\tree)(w),
    \qquad
    \Gn(\tree)(w) = \MvQ{\tree} D(\tree)(w). 
  \end{equation*}
  On the one hand, classical Runge--Kutta theory implies that the local error \eqref{eq:local-err-Bseries} behaves as $\calO(\sw^{p+1})$ for any
  sufficiently smooth $g$.
  On the other hand, the elementary differentials $D(\tree)$ are known to be independent.
  Hence, we obtain that $\Ge(\tree)(w) = \Gn(\tree)(w)$. The
  statement follows because
  the integrand $\Fe(\tree)(\cdot,w) \equiv D(\tree)(w)$ is a constant.
  This yields \eqref{eq:Gn-error}.
  %
\end{proof}

Since the convergence analysis of Lawson methods also employs Taylor expansion, we next study quadrature of monomials. For $\sigma =(\sigma_1,\ldots,\sigma_q) \in [0,1]^q$ and a given vector $\kappa = (\kappa_1, \ldots, \kappa_q )\in\N_0^q$ of non-negative integers, we define as usual
\begin{align*}
\sigma^\kappa = \sigma_1^{\kappa_1}\cdot\ldots\cdot\sigma_q^{\kappa_q}
\end{align*}
and set $\kappa ! = \kappa_1 ! \cdots \kappa_q!$ and $|\kappa| = \kappa_1+\ldots +\kappa_q$. Moreover, we denote the $q$-variate monomial function of degree $|\kappa|$ by
$$
\monomz^\kappa : [0,1]^q\to \R: \sigma \mapsto \sigma^\kappa.
$$
We note that $\monomz^\kappa = \monomz_1^{\kappa_1}\ldots \monomz_q^{\kappa_q}$ and $\monomz^0 = 1$.

It turns out that a multivariate integration (or quadrature)  w.r.t.\ $\tree$ of such monomials corresponds to the integration (or quadrature) of the constant one function w.r.t.\ a particular higher order tree stemming from $\tree$.

\begin{lemmap} \label{lem:tree+mon=bigtree}
Let $\tree \in \calT$ and $\kappa \in \N_0^{\varrho(\tree)}$. Labeling the nodes of $\tree$ with the numbers $\{1,\ldots,\rho(\tree)\}$ (in an arbitrary order) we denote by $\NodeToBush{\kappa}{\tree} \in \calT$ the tree stemming from $\tree$ where $\kappa_j$ leafs are added to its $j$th node. Then $\varrho(\NodeToBush{\kappa}{\tree}) = \varrho(\tree) + \abs{\kappa}$ and
  \begin{align*}
    \MvI{\zeta}{\tree} \monomz^\kappa =
    \MvI{\zeta}{\NodeToBush{\kappa}{\tree}}1,
    \qquad
    \MvQ{\tree} \monomz^\kappa = \MvQ{\NodeToBush{\kappa}{\tree}}1 ,
     \qquad
    \MvQi{i}{\tree} \monomz^\kappa &= \MvQi{i}{\NodeToBush{\kappa}{\tree}}1, \qquad i= 1,\ldots,s.
  \end{align*}
  %
\end{lemmap}

\begin{proof}
  %
%
We prove the lemma by induction on $\varrho(\tree)$. The tree $\tree = \treeroot$ is the unique tree with $\varrho(\tree)=1$. For $k\in\N_0$, we have $\NodeToBush{k}{\treeroot} = \bush{k}$, where $\bush{k} = [\treeroot,\ldots,\treeroot]$  denotes the bush with $k$ leafs. Using $\sigma = \MvI{\sigma}{\treeroot} 1$ and the recursive definition of $\MvI{\zeta}{\tree}$ we obtain
  \begin{align*}
    \MvI{\zeta}{\treeroot} \monomz^k = \int_0^\zeta \sigma^k \dd \sigma = \int_0^\zeta \big ( \MvI{\sigma}{\treeroot}1 \big)^k \dd \sigma = \MvI{\zeta}{\bush{k}} 1 = \MvI{\zeta}{\NodeToBush{k}{\treeroot}} 1.
  \end{align*}
  Analogously, for $i=1,\ldots,\stages$, the simplifying assumptions yield $c_i = \MvQi{i}{\treeroot}1$ and this gives
  \begin{align*}
    \MvQi{i}{\treeroot} \monomz^k =  \sum_{j=1}^{i-1} a_{ij} c_j^k
  = \sum_{j=1}^{i-1} a_{ij} \big ( \MvQi{j}{\treeroot}1\big)^k
  = \MvQi{i}{\bush{k}} 1
  = \MvQi{i}{\NodeToBush{k}{\treeroot}} 1
  \end{align*}
  and $\MvQ{\treeroot} \monomz^k = \MvQ{\NodeToBush{k}{\treeroot}} 1$.

For the induction step,  consider the tree $\tree = [\tree_1,\ldots,\tree_m]$.
We regroup $\sigma$ and $\kappa$  according to the tree structure of $\tau$ as
$\sigma = (\mu,\sigma_{1}, \ldots,\sigma_{m}) \in \R^{\varrho(\tree)}$, $\kappa =(k,\kappa_1,\ldots,\kappa_m)\in \N_0^{\varrho(\tree)}$ with $\mu \in \R$, $k \in \N_0$, $\sigma_j \in \R^{\varrho(\tree_j)}$, $\kappa_j \in \N_0^{\varrho(\tree_j)}$ for $j=1,\ldots,m$. Then it holds $\sigma^\kappa = \mu^{k}\cdot\sigma_1^{\kappa_1}\cdot\ldots\cdot\sigma_m^{\kappa_m}$, $\monomz^\kappa = \monomz_0^k \monomz_1^{\kappa_1}\ldots \monomz_m^{\kappa_m}$, and $\monomz^0 = 1$.
The recursive definition of $\MvQi{i}{\tree}$ and the induction hypothesis imply
  \begin{align*}
    \MvQi{i}{\tree} \monomz^\kappa
    &= \sum_{j=1}^{i-1} a_{ij}  \, \MvQi{j}{\tree_1} \cdots \MvQi{j}{\tree_m} c_j^k \monomz_1^{\kappa_1} \cdots \monomz_m^{\kappa_m} \\
    & = \sum_{j=1}^{i-1} a_{ij}  c_j^k \big(\MvQi{j}{\tree_1}\monomz_1^{\kappa_1}\big) \cdots \big(\MvQi{j}{\tree_m}  \monomz_m^{\kappa_m}\big) \\
    & = \sum_{j=1}^{i-1} a_{ij}  \big(\MvQi{j}{\treeroot}1\big)^k \big(\MvQi{j}{\NodeToBush{\kappa_1}{\tree_1}} 1\big) \cdots \big(\MvQi{j}{\NodeToBush{\kappa_m}{\tree_m}} 1\big) \\
    & = \MvQi{i}{\NodeToBush{\kappa}{\tree}} 1
  \end{align*}
  for $i=1,\ldots,s$, where we used that the sprouted tree can be cast recursively as
  \begin{align*}
    \NodeToBush{\kappa}{\tree} =  [\treeroot,\ldots,\treeroot, \NodeToBush{\kappa_1}{\tree_1},\ldots,\NodeToBush{\kappa_m}{\tree_m}].
  \end{align*}
  The assertion for $\MvI{\zeta}{\tree}$ and $\MvQ{\tree}$ can be shown analogously.
\end{proof}

The following theorem provides a sufficient condition for Lawson
methods being of (stiff) order~$p$. Here, $C^{k,1}(\R^d,X)$ denotes
the space of $k$ time continuously differentiable functions which have
a Lip\-schitz continuous $k$th derivative.

\begin{theorem} \label{thm:lawson-order}
  Let the integrand $\Fe(\tree)$ of $\Ge(\tree)$ satisfy
  \begin{equation}
    \label{eq:lawson-order}
    \Fe(\tree)\bigl(\cdot,u(t)\bigr) \in C^{p-\varrho(\tree),1}
  \bigl([0,1]^{\varrho(\tree)},X\bigr)
  \qquad \text{for all }\tree \in \calT_p,
  \end{equation}
  where $u(t) \in X$ is the solution of \eqref{eq:semi-Lawson}, $0\leq t \leq
  T$.  If the underlying Runge--Kutta method is of (conventional) order
  $p$, then the Lawson method \eqref{eq:lawson} is of (stiff) order
  $p$.
\end{theorem}

\begin{proof}
  Let $\tree \in \calT$ such that $\varrho(\tree) \leq p$.
  We approximate $\Fe(\tree)(\cdot, u_0) \colon [0,1]^{\varrho(\tree)} \to X$ by a multivariate  Taylor
  polynomial of degree $p-\varrho(\tree)$.
  By assumption on $\Fe(\tree)$, the coefficients and the remainder of this Taylor polynomial are bounded.
  Using the linearity of the multivariate integrals and quadrature rules, we have by Lemma~\ref{lem:tree+mon=bigtree}
  \begin{align*}
    \Ge(\tree)(u_0)  -  \Gn(\tree)(u_0)
    & = \MvI{}{\tree} \Fe(\tree)(\cdot,u_0) -   \MvQ{\tree} \Fe(\tree)(\cdot,u_0)
    \\
    & = \sum_{\abs{\kappa} \leq p - \varrho(\tree)} \frac{1}{\kappa !} \Dif{\kappa}(\Fe(\tree))(0,u_0) \left ( \MvI{}{\tree} - \MvQ{\tree} \right) \monomz^\kappa + \calO(\sw^{p- \varrho(\tree)+1})
    \\
    & = \sum_{\abs{\kappa} \leq p- \varrho(\tree)} \frac{1}{\kappa !} \Dif{\kappa}(\Fe(\tree))(0,u_0) \left ( \MvI{}{\NodeToBush{\kappa}{\tree}} 1  - \MvQ{\NodeToBush{\kappa}{\tree}} 1 \right)  + \calO(\sw^{p- \varrho(\tree)+1}).
  \end{align*}
  Here we used $\norm{\Dif{\kappa} \Fe(\tree) } = \calO(\sw^{\abs{\kappa}})$ to bound the remainder term.
  Since the Runge--Kutta method is of order $p$, the claim now follows from  $\varrho(\NodeToBush{\kappa}{\tree}) = \varrho(\tree) + \abs{\kappa}\leq p$ and Corollary~\ref{cor:rkm-qf} which implies $\MvI{}{\NodeToBush{\kappa}{\tree}} 1  = \MvQ{\NodeToBush{\kappa}{\tree}} 1$.
\end{proof}

%

This result now allows us to prove an error bound for Lawson methods
which is uniform for all problems \eqref{eq:semi-Lawson} with $A$ satisfying
Assumption~\ref{ass:noframe}.

\begin{theorem} \label{thm:lawson-error} Let $u$ be the solution of
  \eqref{eq:semi-Lawson} and let the assumptions of
  Theorem~$\ref{thm:lawson-order}$ be satisfied.  If the underlying
  Runge--Kutta method is of (conventional) order $p$, then there
  exists $\sw_0>0$ such that for all $0 <\sw \leq \sw_0$ sufficiently
  small,
  \begin{equation*}
    \norm{u(t_n) - u_n} \leq C \sw^p, \qquad t_n = n\sw \leq T,
  \end{equation*}
  where $C$ and $\sw_0$ are independent of $n$, $\sw$, and $A$.
\end{theorem}
\begin{proof}
   We  define a norm by
   \begin{equation*}
     \norm{v}_\star = \sup_{t \in \R} \norm{\ee^{-tA} v}.
   \end{equation*}
   This norm is equivalent to $\norm{\cdot}$ 
   and we have in the corresponding operator norm
   \begin{equation}
     \label{eq:A-contraction} \norm{\ee^{-tA}}_\star \leq 1, \qquad \text{for
       all } t \in \R.
   \end{equation}
   If $-A$ only generates a bounded semigroup, then taking the
   supremum only over $t \geq 0$ shows that $-A$ generates
   a contraction semigroup.

  By assumption, $g$ is locally Lipschitz continuous. Then
  \eqref{eq:A-contraction} and Theorem~\ref{thm:lawson-expand}
  show that the Lawson method is
  locally Lipschitz with respect to the initial value with a Lipschitz
  constant of size $1+\calO(\sw)$. This implies the required
  stability.

  The error bound follows in a standard way using Lady Windermere's
  fan.
\end{proof}

\section{Regularity conditions and applications} \label{sec:applications}

It remains to discuss the regularity conditions~\eqref{eq:lawson-order} and to give some applications. We first examine the conditions for orders one and two, respectively. The extension to higher orders is a tedious but straightforward exercise. It turns out that these regularity conditions can all be expressed in terms of commutators, very much like in the case of splitting methods.

In order to obtain simple sufficient conditions, we replace the space $C^{k,1}(\Omega,X)$ in condition~\eqref{eq:lawson-order} by the subspace of $k+1$ times partially differentiable functions with uniformly bounded partial derivatives on $\Omega$ in the following discussion. This is also justified by the fact that Lipschitz continuous functions are almost everywhere differentiable (Rademacher's theorem).

\subsection{Condition for order one} \label{subsec:order-1}

Since $p=\varrho(\tree)=1$, we only have to consider the tree $\tree=\,\treeroot\,$ in~\eqref{eq:lawson-order}.
Differentiating
\begin{equation}   \label{eq:psiroot}
  \Fe(\treeroot)(\sigma,w) = \ee^{-(1-\sigma)\sw A} g_\sigma(w)
              = \ee^{-(1-\sigma)\sw A} g\big(\ee^{-\sigma \sw A}w\big)
\end{equation}
with respect to $\sigma$ yields
\begin{equation}\label{eq:diff1}
\begin{aligned}
  \partial_\sigma \Fe(\treeroot)(\sigma,w) &=
    \sw  \ee^{-(1-\sigma)\sw A} \Bigl(A g_\sigma(w)-
    g_\sigma'(w) A \ee^{-\sigma \sw A} w  \Bigr)\\
    &=\sw  \ee^{-(1-\sigma)\sw A} [F_A,g]\bigl(\ee^{-\sigma \sw A} w\bigr),
\end{aligned}
\end{equation}
where $[F_A,g]$ denotes the Lie commutator of $g$ and $F_A(w) = Aw$, defined as
\begin{equation}
  \label{eq:def-liecomm}
   [F_A,g](w) = F_A'(w)g(w) - g'(w)F_A(w) = Ag(w) - g'(w)Aw .
\end{equation}
From this calculation, we conclude the following result. If the bound
\begin{equation}\label{eq:c1-1}
\sup_{\ 0\le\sigma\le 1\ }\sup_{0\le t\le T}\norm{\ee^{-(1-\sigma)\sw A}\comm{F_A,g}\bigl(\ee^{-\sigma \sw A} u(t)\bigr)} \le C
\end{equation}
holds with a constant $C$ that is allowed to depend on $\CE$, then a Lawson method of non-stiff order \emph{one} has also \emph{stiff} order \emph{one}.

\subsection{Conditions for order two} \label{subsec:order-2}

Stiff order two is achieved if we require the following two regularity conditions
\begin{equation*}
  \Fe(\treeroot)(\cdot,u(t)) \in C^2([0,1],X)
  \qquad \text{and} \qquad
  \Fe(\treeba{1}\,) (\cdot,u(t))\in C^1([0,1]^2,X).
\end{equation*}
We commence with the first condition and exploit the fact that
$\partial_\sigma \Fe(\treeroot)(\sigma,w)$ is of exactly the same form
as \eqref{eq:psiroot} with $g$ replaced by the vector field
$\comm{F_A,g}$.  Hence from \eqref{eq:diff1} we have
\begin{align*}
  \partial^2_\sigma \Fe(\treeroot)(\sigma,w)
  & =     \sw ^2 \ee^{-(1-\sigma)\sw A} \bigl[F_A,[F_A,g]\bigr]\bigl(\ee^{-\sigma \sw A}w \bigr).
\end{align*}
Therefore, the bound
\begin{equation}\label{eq:c1-2}
   \sup_{\ 0\le\sigma\le 1\ }\sup_{0\le t\le T}\norm{\ee^{-(1-\sigma)\sw A}\bigl[F_A,\comm{F_A,g}\bigr]\bigl(\ee^{-\sigma \sw A} u(t)\bigr)}\le C
\end{equation}
should hold with a constant $C$ that is independent of $\|A\|$.

Next, we move to the second condition. Differentiating
\begin{equation*}
    \Fe(\treeba{1}\,)(\sigma_1,\sigma_2,w)
    =  \ee^{-(1-\sigma_1) \sw A}  g_{\sigma_1}'(w)
       \ee^{-(\sigma_1-\sigma_2) \sw A}  g_{\sigma_2}(w)
\end{equation*}
with respect to $\sigma_1$ and $\sigma_2$ yields
\begin{equation*}
\begin{aligned}
  \partial_{\sigma_1}  \Fe(\treeba{1}\,)(\sigma_1,\sigma_2,w)
  &=\sw  \ee^{-(1-\sigma_1) \sw A}  \Bigl( A g_{\sigma_1}'(w) \ee^{-(\sigma_1-\sigma_2) \sw A}g_{\sigma_2}(w)\\
    &\hspace{5mm}-  g_{\sigma_1}''(w)\bigl(  \ee^{-\sigma_1 \sw A} Aw
    ,\ee^{-(\sigma_1-\sigma_2) \sw A} g_{\sigma_2}(w)\bigr) -
    g_{\sigma_1}'(w) A \ee^{-(\sigma_1-\sigma_2) \sw A}
     g_{\sigma_2}(w) \Bigr)\\
     &= \sw  \ee^{-(1-\sigma_1) \sw A}
     \comm{F_A,g}'\big(\ee^{-\sigma_1\sw A}w\big)  \ee^{-(\sigma_1-\sigma_2) \sw A} g_{\sigma_2}(w) ,
\end{aligned}
\end{equation*}
since by definition \eqref{eq:def-liecomm} the derivative of the
commutator satisfies
\begin{equation} \label{eq:comm-der1}
 \comm{F_A,g}' (w) v  = \frac{d}{dw}   \big( \comm{F_A,g} \big)(w) v
    = A g'(w) v- g''(w) (Aw,v) - g'(w) Av .
\end{equation}
Moreover, we have
\begin{equation*}
  \partial_{\sigma_2}  \Fe(\treeba{1}\,)(\sigma_1,\sigma_2,w)
   = \sw  \ee^{-(1-\sigma_1) \sw A}
   g_{\sigma_1}'(w)  \ee^{-(\sigma_1-\sigma_2) \sw A} \comm{F_A,g} \bigl(\ee^{-\sigma_2 \sw A} w \bigr),
\end{equation*}
respectively. From these two relations, we infer that the bounds
\begin{subequations}\label{eq:c2}
\begin{align}
\sup_{\ 0\le\sigma_1,\sigma_2\le 1\ }\sup_{0\le t\le T}
 \Big \| \ee^{-(1-\sigma_1) \sw A}
     \comm{F_A,g}'\big(\ee^{-\sigma_1\sw A}u(t)\big)
  \ee^{-(\sigma_1-\sigma_2) \sw A} g_{\sigma_2}(u(t))  \Big \| & \le C,  \label{eq:c2-1a} \\
\sup_{\ 0\le\sigma_1,\sigma_2\le 1\ }\sup_{0\le t\le T}\norm{\ee^{-(1-\sigma_1) \sw A} g_{\sigma_1}'(u(t)) \ee^{-(\sigma_1-\sigma_2) \sw A} \comm{F_A,g}\bigl(\ee^{-\sigma_2 \sw A} u(t) \bigr)}&\le C \label{eq:c2-1b}
\end{align}
\end{subequations}
should hold with a constant $C$ that is independent of $\|A\|$.

From the above calculations, we conclude the following result. If the conditions \eqref{eq:c1-1}, \eqref{eq:c1-2}, and \eqref{eq:c2} hold with a constant $C$ that does not depend on $\|A\|$, then a Lawson method of
non-stiff order \emph{two} has also \emph{stiff} order \emph{two}.

\subsection{Conditions for higher order} \label{subsec:higherorder}

The following lemma provides the formulas to derive
the order conditions for order larger than two in a systematic way.

\begin{lemmap} \label{lem:pnsigmabusch}
  Let $m\geq 1$. 
  \begin{enumerate}
  \item For $\tree = \,\treeroot\,$ we have
    \begin{equation*}
  \partial^m_\sigma \Fe_\zeta(\treeroot)(\sigma,w)
   =    \sw^m \ee^{-(\zeta-\sigma)\sw A}
   [F_A,g]_m\bigl(\ee^{-\sigma \sw A}w \bigr),
    \end{equation*}
    where $ \bigl[F_A,g\bigr]_{m+1} = \bigl[F_A, [F_A,g]_{m}\bigr]$ with  $[F_A,g]_1 = [F_A,g]$ denotes the $(m+1)$-fold
    commutator.
  \item For $\tree= [\tree_1,\ldots,\tree_k]$ we have
    \begin{align*}
      \partial_\sigma^m
       \Fe_\zeta(\tree)(\sigma,\sigma_1,\ldots,\sigma_k,w)
          =  \sw \ee^{-(\zeta-\sigma) \sw A} &\bigl[F_A,g\bigr]_m^{(k)}
            \bigl(\ee^{-\sigma \sw A} w \bigr)
            \bigl(\Fe_\sigma(\tree_1)(\sigma_1,w),\ldots,\Fe_\sigma(\tree_k)(\sigma_k,w)\bigr)
    \end{align*}
    for $\sigma \in \R$ and $\sigma_j \in \R^{\varrho (\tree_j)}$, $j=1,\ldots,k$.
  \end{enumerate}
\end{lemmap}
\begin{proof}
Both parts are proved by induction on $m$.

(a)  For $m=1$ the statement was proved in \eqref{eq:diff1}. The induction step
is proved by the same arguments as were used for $m=2$ above.

(b)   To prove the statement for $m=1$, we first note that for
    $\tree= [\tree_1,\ldots,\tree_k]$ the integrand of
    $\Ge_\zeta(\tree)$ is given recursively as
    \begin{equation*}
      \Fe_\zeta(\tree)(\sigma,\sigma_1,\ldots,\sigma_k,w) =
        \ee^{-(\zeta-\sigma) \sw A}
     g_\sigma^{(k)}\bigl(w \bigr)
            \bigl(\Fe_\sigma(\tree_1)(\sigma_1, w),\ldots,\Fe_\sigma(\tree_k)(\sigma_k,w)\bigr).
    \end{equation*}
    Since $\partial_\eta \Fe_\eta(\tree) = -\sw A\Fe_\eta(\tree)$ for any tree $\tree$, we obtain
    \begin{align*}
      \partial_\sigma
      \Fe_\zeta(\tree)(\sigma,\sigma_1,\ldots,\sigma_k,w)
         = ~ &\sw  \ee^{-(\zeta-\sigma) \sw A}
             \Bigl( A g_\sigma^{(k)}(w)
          \bigl(\Fe_\sigma(\tree_1)(\sigma_1,
          w),\ldots,\Fe_\sigma(\tree_k)(\sigma_k,w)\bigr)\\
             &
               -g_\sigma^{(k+1)}(w) \big(A \ee^{-\sigma \sw A} w,
             \Fe_\sigma(\tree_1)(\sigma_1, w),\ldots,\Fe_\sigma(\tree_k)(\sigma_k,w)\big)\\
         &   - g_\sigma^{(k)}(w)
            \big( A \Fe_\sigma(\tree_1)(\sigma_1, w),
           \Fe_\sigma(\tree_2)(\sigma_2,w)\ldots,\Fe_\sigma(\tree_k)(\sigma_k,w)\big)\\
         & - \ldots \\
         &   - g_\sigma^{(k)}(w)
            \big(  \Fe_\sigma(\tree_1)(\sigma_1, w),\ldots,
            \Fe_\sigma(\tree_{k-1})(\sigma_{k-1}, w),
           A\Fe_\sigma(\tree_k)(\sigma_k,w)\big) \Big).
    \end{align*}
    On the other hand, by definition \eqref{eq:def-liecomm}, we have
    \begin{equation}
      \label{eq:comm-der}
      \comm{F_A,g}^{(k)}(w)(v_1,\ldots,v_k) =
             A g^{(k)}(w)(v_1,\ldots,v_k) - \frac{d^k}{dw^k} \big( g'(w) Aw\big) (v_1,\ldots,v_k).
    \end{equation}
    Using induction on $k$ it is easy to see that
    \begin{equation}
      \label{eq:comm-derm-help}
      \begin{aligned}
      \frac{d^k}{dw^k} \big( g'(w) Aw\big) (v_1,\ldots,v_k)
         = ~&g^{(k+1)}(w)(Aw,v_1,\ldots,v_k) \\
      & +     g^{(k)}(w)(Av_1,v_2,\ldots,v_k)
      + \ldots
      +     g^{(k)}(w)(v_1,\ldots,Av_k).
      \end{aligned}
    \end{equation}
    This proves the claim for $m=1$. If it holds for some
    $m\geq 1$ then it does also for $m+1$, since the same calculation
    can be done with $\comm{F_A,g}^{(k)}$ in the role of $g^{(k)}$.
\end{proof}

The lemma thus shows that all derivatives arising in the order
conditions can be obtained recursively from the tree structure. Moreover, only commutators, iterated commutators and their
derivatives appear.

\subsection{Specialisation to linear problems} \label{subsec:linear-Lawson}

For the linear evolution equation
\begin{equation*}
  u' + A u = B u, \qquad u(0)=u_0
\end{equation*}
with bounded operator $B$ on $X$, the above conditions \eqref{eq:c1-1}, \eqref{eq:c1-2}, and \eqref{eq:c2} simplify a bit.
Having $g(u) = B u$, the Lie commutator coincides with the operator commutator of $A$ and $B$
\begin{align*}
  \comm{F_A,g}(w) = A B w - B A w = \comm{A,B} w .
\end{align*}
A first-order Lawson method is of \emph{stiff} order one if
\begin{subequations}
\begin{equation}  \label{eq:lawson-o1-linear}
\sup_{\ 0\le\sigma\le 1\ }\sup_{0\le t\le T}\norm{ \ee^{-(1-\sigma)\sw A }  \comm{A,B} \ee^{-\sigma \sw A} u(t)}   \leq C.
\end{equation}
For second order, the conditions read
\begin{align}
\sup_{\ 0\le\sigma\le 1\ }\sup_{0\le t\le T}\norm{\ee^{-(1-\sigma)\sw A} \comm{A,\comm{A,B}} \ee^{-\sigma \sw A} u(t)}&\le C,\\
\sup_{\ 0\le\sigma_2\le\sigma_1\le 1\ }\sup_{0\le t\le T}\norm{ \ee^{-(1-\sigma_1) \sw A}  \comm{A,B}
       \ee^{-(\sigma_1-\sigma_2) \sw A}  B \ee^{-\sigma_2 \sw A} u(t) }& \leq C,\\
\sup_{\ 0\le\sigma_2\le\sigma_1\le 1\ }\sup_{0\le t\le T}\norm{ \ee^{-(1-\sigma_1) \sw A} B
   \ee^{-(\sigma_1-\sigma_2) \sw A} \comm{A,B}   \ee^{-\sigma_2 \sw A} u(t) } & \leq C.
\end{align}
\end{subequations}
We recall that such conditions also arise in the analysis of splitting methods, cf.\ \cite{JahL00}.

Using Lemma~\ref{lem:pnsigmabusch}, the above analysis can easily be
generalized to higher order, since for linear problems, only long
trees have to be considered.  For all other trees, which have at least
one node with two branches, the integrand $\Fe$ vanishes.

\subsection{Nonlinear Schr\"{o}dinger equations} \label{subsec:schroedinger}

For the time discretization of nonlinear Schr\"{o}dinger equations
\begin{equation}\label{eq:nls}
u' = \ii \bigl( \Delta u + f(|u|^2)u\bigr),
\end{equation}
split-step methods are commonly viewed as the method of choice. In recent years, however, exponential integrators have been considered as a viable alternative for the solution of \eqref{eq:nls}. For instance, \cite{BesDLV17} studied exponential integrators in the context of Bose--Einstein condensates; \cite{CanG15P} and \cite{BalFMMTP16} reported favorable results for Lawson integrators of the form as discussed in this paper. Rigorous convergence results, however, are still missing for these methods.

As an application of our analysis, we will use the above regularity conditions \eqref{eq:c1-1}, \eqref{eq:c1-2}, and \eqref{eq:c2} to verify second-order convergence of Lawson methods. We refrain from any particular space discretization and argue in an abstract Hilbert space framework. Note, however, that our reasoning carries over to spatial discretizations (by spectral methods, e.g.) without any difficulty.

For this purpose, we consider \eqref{eq:nls} with periodic boundary conditions on the $d$ dimensional torus and smooth potential. Then it is well known (see, e.g., \cite[Thm.~4.1]{Kat95}) that the problem is well posed in $H^m$ for $m>d/2$. The regularity of an initial value $u_0\in H^m$ is thus preserved along the solution. Henceforth we choose $m>d/2$.

Second-order Strang splitting for~\eqref{eq:nls}
with $f(u) = \pm u$ was rigorously analysed in \cite{Lub08}. There it was shown that commutator relations similar to our conditions \eqref{eq:c1-1}, \eqref{eq:c1-2}, and \eqref{eq:c2} play a crucial role in the convergence proof for Strang splitting. The analysis given here shows that Lawson methods converge under the same regularity assumptions as splitting schemes. This will be worked out now in detail for first and second-order methods.

Let $A=- \ii\, \Delta$ and $g(u)= \ii \beta |u|^2 u$, $\beta \in \R$, i.e.~$f=\beta I$.
By
\begin{align*}
  g(u+w) = \ii\beta (u+w)^2( \overline{u+w} )
                = g(u) + \ii\beta( u^2 \overline{w}
                   +  2 u  \overline{u} w )  + \calO(\abs{w}^2),
\end{align*}
the Fr\'echet derivative of $g$ is given by
\begin{equation*}
  g'(u) w = \ii \beta (u^2 \overline{w} + 2 \abs{u}^2 w).
\end{equation*}
The first commutator $\comm{F_A,g}$ then takes the form
\begin{align}
\comm{F_A,g}(u) &= Ag(u) - g'(u)Au \nonumber \\
& = - \ii \Delta g(u) + g'(u) (i \Delta u) \nonumber \\
&= \beta \nabla\cdot\nabla(u^2 \overline u)
    +\ii \beta\big(u^2 \overline{(\ii\Delta u)} + 2 \abs{u}^2 \ii \Delta u\big) \nonumber \\
&= \beta \nabla\cdot\bigl(2u\,\overline u \,\nabla u +
  u^2\nabla\overline u\bigr)
+ \beta \big( u^2\Delta \overline u
- 2 \abs{u}^2 \Delta u \big) \nonumber \\
&= \beta\bigl( 2 \overline u \,\nabla u\cdot\nabla u + 2 u\,\nabla
  u\cdot\nabla\overline u
  + 2 \abs{u}^2 \Delta u + 2 u \nabla u \cdot \nabla \overline{u}
   + 2 u^2 \Delta \overline{u} - 2 \abs{u}^2 \Delta u\big) \nonumber \\
&= 2\beta\bigl( \overline u \,\nabla u\cdot\nabla u + 2 u\,\nabla u\cdot\nabla\overline u +u^2\Delta\overline u\bigr).
 \label{eq:comm-FA-g}
\end{align}

We next show that   the commutator can be bounded
in $H^m$ if the solution is in $H^{m+2}$ for $m\geq 0$.

\begin{lemmap}  \label{lem:commutator-bound}
   Let $\Omega \subset \R^d$, $d\leq 3$, be a bounded Lipschitz
   domain. Then there exists a constant
  $C$ which only depends on $\Omega$ and $d$ such that
  \begin{align}  \label{eq:commutator-bound}
    \norm{\comm{F_A,g}u}_m \leq C \norm{u}_{m+2}^3.
  \end{align}
\end{lemmap}
\begin{proof}
  Note that by the Sobolev embedding theorem we have the following bounds
  \begin{subequations} \label{eq:sob-est-product}
      \begin{align}
        \norm{uvw}_0 &\leq C \norm{u}_1 \norm{v}_1 \norm{w}_1,
        \label{eq:sob-est-product-a}\\
        \norm{uvw}_0 &\leq C \norm{u}_0 \norm{v}_2 \norm{w}_2,
        \label{eq:sob-est-product-b}\\
        \norm{uvw}_1 &\leq C \norm{u}_1 \norm{v}_2 \norm{w}_2,
        \label{eq:sob-est-product-c}\\
        \norm{uvw}_m &\leq C \norm{u}_m \norm{v}_m \norm{w}_m,
                       \quad m \geq 2,
        \label{eq:sob-est-product-d}
      \end{align}
  \end{subequations}
  cf.\ \cite[Section~8]{Lub08}.

  For $m=0$, the bound \eqref{eq:commutator-bound} follows from using
  \eqref{eq:sob-est-product-a} for the first two terms
  and \eqref{eq:sob-est-product-b} for the last one in the
  explicit  expression \eqref{eq:comm-FA-g} of $\comm{F_A,g}$. For $m=1$ we apply
  \eqref{eq:sob-est-product-c} to all terms and for $m\geq 2$
  the bound follows from  \eqref{eq:sob-est-product-d}.
\end{proof}

For Lawson methods, a first-order
convergence bound in $H^m$ thus requires $H^{m+2}$ regularity of the
exact solution, which is the same regularity as required for the first-order Lie splitting.

For second-order methods, one has to estimate the double commutator
$\comm{F_A,\comm{F_A,g}}$. A simple calculation shows that a bound in
$H^m$ requires $H^{m+4}$ regularity of the exact solution. This
situation is exactly the same as for second-order Strang splitting
(see \cite{Lub08}). Using \eqref{eq:comm-FA-g} we conclude that the
derivative of the commutator $\comm{F_A,g}$ can be expressed as
\begin{align*}
\comm{F_A,g}'(u)w =
   2 \beta \big( \overline{w} \nabla u \cdot \nabla u
   &+ 2  \overline{u} \nabla u \cdot \nabla w
   + 2  {w} \nabla u \cdot \nabla \overline{u}
   + 2 u \nabla w \cdot \nabla \overline{u}\\
   &+ 2 u \nabla u \cdot \nabla \overline{w}
   +  u^2 \Delta \overline{w}
   +  2 uw \Delta \overline{u}\bigr).
\end{align*}
This commutator can again be bounded in $H^m$ for $u,w\in H^{m+2}$. We thus conclude that Lawson methods require the same regularity for second-order convergence as Strang splitting.

\subsection{Numerical examples}  \label{subsec:numexamples}

Lawson methods exhibit a strong order reduction, in general. For
particular problems, however, they show full order of convergence (see
\cite{KasT05}, \cite{BesDLV17}, \cite{CanG15P}, \cite{BalFMMTP16}, and
\cite{MonB16}). Most of the problems considered in these papers result
from space discretizations of partial differential equations posed
with periodic boundary conditions.

  After space discretization (by finite differences, finite elements,
  or spectral methods) the evolution equation \eqref{eq:semi-Lawson}
  becomes an ordinary differential equation
  \begin{equation}  \label{eq:semi-Lawson-dx}
    \bfu'(t) + A_N \bfu(t) = g_N\bigl(\bfu(t) \bigr), \qquad
    \bfu(0) = \bfu_0.
  \end{equation}
  with a matrix $A_N \in \C^{N \times N}$ and a
  discretization $g_N: \C^N \to\C^N$ of
  $g$, where $N$ denotes the
  employed degrees of freedom.  In order to satisfy
  Assumption~\ref{ass:noframe} the space discretization is required to
  provide matrices $A_N$ such that
  \begin{equation}\label{eq:mat-bound-dx}
    \left\| \ee^{-tA_N}\right\| \le \CE
  \end{equation}
  holds with a constant $\CE$ being uniform in $N$ and $t\in \R$.

In the previous sections we showed
that full order of convergence is only guaranteed if certain
regularity conditions are satisfied. The aim of the following
numerical examples is to show that order reduction can also be
verified numerically, if some of these regularity assumptions are
violated. In fact, such order reductions can even be observed for
linear problems. Hence we resign from presenting numerical examples
for semilinear problems here. Numerous such examples can be found in
the literature mentioned above.  We also restrict ourselves to the
first order schemes covered by our analysis, the exponential Euler and
the Lawson Euler method, since they already show interesting (and different)
convergence behavior.

We consider the linear Schr\"odinger equation
\begin{equation}  \label{eq:linschroedi}
 u_t = i u_{xx} + i f(x) u, \qquad x \in [-\pi,\pi], \quad u(0,\cdot) = u_0,
\end{equation}
with periodic boundary conditions and discretize it using a Fourier
spectral method on an equidistant grid.
Let $N$ be even and denote by $\calF_N$ the discrete Fourier matrix. Then
matrix $A_N$ is given as
\begin{equation*}
 A_N= i \calF_N^{-1} D_N^2 \calF_N, \qquad \text{where} \qquad  D_N = 
 \text{diag}(-\tfrac{N}{2}+1,-\tfrac{N}{2}+2,\ldots,\tfrac{N}{2}),
\end{equation*}
and
\begin{equation*}
  g_N (\bfu) = B_N \bfu, \qquad B_N = i  \,
        \text{diag}\bigl(f(x_{-N/2+1}),\ldots,f(x_{N/2})\bigr),
    \qquad
    x_m = m \tfrac{2\pi}{N}.
\end{equation*}
With this notation, the exact solution of \eqref{eq:semi-Lawson-dx} is
given by
\begin{equation}
  \label{eq:schroedi-N-sol}
  \bfu(t) = \ee^{t(-A_N+B_N)} \bfu_0.
\end{equation}

\begin{example}
  The aim of the first example is to explain that the concept of
  regularity is relevant even in the ODE context.  In order to show what
  regularity means here, for each $N=2^7,\ldots,2^{12}$ we choose a
  a regularity parameter $\alpha \geq 0$ and a
  vector $\bfr = (r_m)_{m=-N/2+1}^{N/2} \in \C^N$ of Fourier coefficients whose
  entries contain random numbers uniformly distributed in the unit
  disc.  Then we define an initial function as the trigonometric
  polynomial
  \begin{equation}
    \label{eq:w-rand-reg}
    \widetilde u_{0;N}(x) =  \sum_{m=-N/2+1}^{N/2}
    \nu_m  e^{imx},
    \qquad
      \nu_m = \frac{r_m}{(1+m^2)^{\tfrac12\bigl(\tfrac12+\alpha+\epsilon\bigr)}},
    \qquad \epsilon = 10^{-6}.
  \end{equation}
  In the limit $N\to\infty$, this sequence of trigonometric
  polynomials converges to a function in the Sobolev space
  $\Hper^\alpha = \Hper^{\alpha}((-\pi,\pi))$ equipped with the norm
  \begin{equation*}
    \norm{u}_\alpha^2 =
    2\pi  \sum_{m \in \Z}
    (1+m^2)^{\alpha} \abs{\nu_m}^2
    \qquad
    \text{for}
    \qquad
    u = \sum_{m \in \Z} \nu_m e^{imx}.
  \end{equation*}
  For $\alpha=0$ we have the standard $L^2$ norm
  \begin{equation*}
    \norm{u}_0^2 =   \int_{-\pi}^\pi \abs{u(x)}^2 \dd x.
  \end{equation*}
  Then we define an initial vector $\bfu_0 \in \C^N$
  for \eqref{eq:semi-Lawson-dx}  corresponding
  to a function $u_0 \in \Hper^\alpha$ by setting the $j$th component
  as
  \begin{equation}
    \label{eq:u0-rand-reg}
     (\bfu_{0})_j = u_{0;N}(x_j),
     \qquad j=- \tfrac{N}{2}+1,\ldots, \tfrac{N}{2},
   \end{equation}
   where $u_{0;N} = {\widetilde u}_{0;N}/\norm{{\widetilde u}_{0;N}}_0$
   has unit $L^2$ norm.
   The discrete Sobolev norms in $\C^N$ corresponding to $\norm{\cdot}_\alpha$
   can be computed via
  \begin{equation*}
    \norm{\bfu}_{\alpha,N}^2 = 2\pi
        \norm{ (I+D_N^2)^{\alpha/2} \calF_N \bfu}_{\C^N}^2,
  \end{equation*}
  where $\norm{\cdot}_{\C^N}$ denotes the Euclidean norm in $\C^N$.
  This yields $\norm{\bfu_0}_{\alpha,N} =  \norm{u_{0;N}}_{\alpha}$.
%

  In Figure~\ref{fig:regularity} we plot $\norm{\bfu_0}_{\mu,N}$ for
  different values of $\mu$ over the number of Fourier modes $N$.
  The three graphs clearly show that $\norm{\bfu_0}_{\mu,N}$ is bounded
  independently of the number $N$ of Fourier modes only for
  $\mu \leq \alpha$.  This corresponds to the continuous case, where
  obviously, the Sobolev norm $\norm{u}_{\mu}$ is bounded for all
  functions $u \in \Hper^{\alpha}$ for $\mu \leq \alpha$.
 \end{example}

The example clearly shows that regularity of the corresponding continuous function is crucial to obtain error bounds which do not deteriorate in the limit $N\to \infty$.
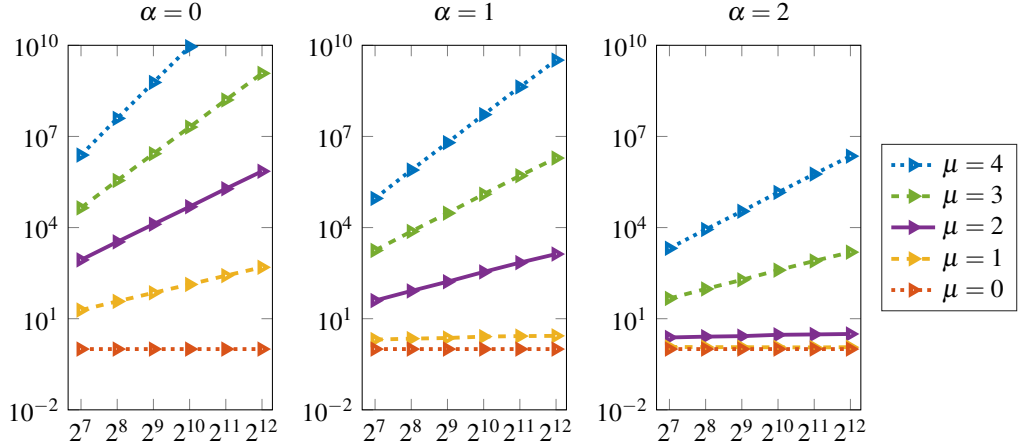
\begin{figure}[tb]
  \centering
\scalebox{0.97}{
%
%
%
\definecolor{mycolor1}{rgb}{0,0.447,0.741}%
\definecolor{mycolor2}{rgb}{0.466,0.674,0.188}%
\definecolor{mycolor3}{rgb}{0.494,0.184,0.556}%
\definecolor{mycolor4}{rgb}{0.929,0.694,0.125}%
\definecolor{mycolor5}{rgb}{0.85,0.325,0.098}%
\begin{tikzpicture}

\begin{axis}[%
width=2.8cm,
height=5cm,
scale only axis,
xmode=log,
xmin=100,
xmax=5000,
xtick = {128,256,512,1024,2048,4096},
xticklabels = {$2^7$,$2^8$,$2^9$,$2^{10}$,$2^{11}$,$2^{12}$},
ymode=log,
ymin=0.01,
ymax=1e10,
yminorticks=true,
title={$\alpha = 0$},
legend style={anchor=west,legend cell align=left,align=left,draw=white!15!black},
]
\addplot [color=mycolor1,dotted,line width=1.5pt,mark=triangle,mark options={solid,rotate=270}]
  table[row sep=crcr]{%
128 2436362.31083268\\
256 39255103.2206164\\
512 597971795.346574\\
1024 8997775657.03496\\
2048 140903254203.141\\
4096 2116176489023.03\\
};
\addplot [color=mycolor2,dashed,line width=1.5pt,mark=triangle,mark options={solid,rotate=270}]
  table[row sep=crcr]{%
128 44438.7915226698\\
256 354009.279374927\\
512 2693811.884962\\
1024 20279862.7645762\\
2048 158818468.702048\\
4096 1193527997.78174\\
};
\addplot [color=mycolor3,solid,line width=1.5pt,mark=triangle,mark options={solid,rotate=270}]
  table[row sep=crcr]{%
128 861.528153630658\\
256 3378.42551675343\\
512 12875.623163231\\
1024 48546.749679145\\
2048 189789.774482608\\
4096 713760.606078139\\
};
\addplot [color=mycolor4,dashed,line width=1.5pt,mark=triangle,mark options={solid,rotate=270}]
  table[row sep=crcr]{%
128 19.2668033643005\\
256 37.1451923071092\\
512 71.1571903756873\\
1024 134.480379421228\\
2048 262.004969888248\\
4096 492.83304008339\\
};
\addplot [color=mycolor5,dotted,line width=1.5pt,mark=triangle,mark options={solid,rotate=270}]
  table[row sep=crcr]{%
128 1\\
256 1\\
512 1\\
1024 1\\
2048 1\\
4096 1\\
};
\end{axis}
\end{tikzpicture}%
%
%
%
\definecolor{mycolor1}{rgb}{0,0.447,0.741}%
\definecolor{mycolor2}{rgb}{0.466,0.674,0.188}%
\definecolor{mycolor3}{rgb}{0.494,0.184,0.556}%
\definecolor{mycolor4}{rgb}{0.929,0.694,0.125}%
\definecolor{mycolor5}{rgb}{0.85,0.325,0.098}%
\begin{tikzpicture}

\begin{axis}[%
width=2.8cm,
height=5cm,
scale only axis,
xmode=log,
xmin=100,
xmax=5000,
xtick = {128,256,512,1024,2048,4096},
xticklabels = {$2^7$,$2^8$,$2^9$,$2^{10}$,$2^{11}$,$2^{12}$},
ymode=log,
ymin=0.01,
ymax=1e10,
yminorticks=true,
title={$\alpha = 1$},
legend style={anchor=west,legend cell align=left,align=left,draw=white!15!black},
]
\addplot [color=mycolor1,dotted,line width=1.5pt,mark=triangle,mark options={solid,rotate=270}]
  table[row sep=crcr]{%
128 91214.588258261\\
256 780872.29248583\\
512 6237599.77399377\\
1024 52221676.4551713\\
2048 424904216.172977\\
4096 3250104098.45532\\
};
\addplot [color=mycolor2,dashed,line width=1.5pt,mark=triangle,mark options={solid,rotate=270}]
  table[row sep=crcr]{%
128 1768.36347510107\\
256 7452.11787362748\\
512 29813.8799451216\\
1024 125010.345687489\\
2048 507765.098248543\\
4096 1943646.29522059\\
};
\addplot [color=mycolor3,solid,line width=1.5pt,mark=triangle,mark options={solid,rotate=270}]
  table[row sep=crcr]{%
128 39.5468345495179\\
256 81.9347208154955\\
512 164.766544049782\\
1024 346.293806088822\\
2048 700.970216333249\\
4096 1342.03695799867\\
};
\addplot [color=mycolor4,dashed,line width=1.5pt,mark=triangle,mark options={solid,rotate=270}]
  table[row sep=crcr]{%
128 2.05258930616349\\
256 2.20579611321097\\
512 2.31552908679878\\
1024 2.57505078123061\\
2048 2.67540809104587\\
4096 2.72310670926527\\
};
\addplot [color=mycolor5,dotted,line width=1.5pt,mark=triangle,mark options={solid,rotate=270}]
  table[row sep=crcr]{%
128 1\\
256 1\\
512 1\\
1024 1\\
2048 1\\
4096 1\\
};
\end{axis}
\end{tikzpicture}%
%
%
%
\definecolor{mycolor1}{rgb}{0,0.447,0.741}%
\definecolor{mycolor2}{rgb}{0.466,0.674,0.188}%
\definecolor{mycolor3}{rgb}{0.494,0.184,0.556}%
\definecolor{mycolor4}{rgb}{0.929,0.694,0.125}%
\definecolor{mycolor5}{rgb}{0.85,0.325,0.098}%
\begin{tikzpicture}

\begin{axis}[%
width=2.8cm,
height=5cm,
scale only axis,
xmode=log,
xmin=100,
xmax=5000,
xtick = {128,256,512,1024,2048,4096},
xticklabels = {$2^7$,$2^8$,$2^9$,$2^{10}$,$2^{11}$,$2^{12}$},
ymode=log,
ymin=0.01,
ymax=1e10,
yminorticks=true,
title={$\alpha = 2$},
legend style={at={(1.1,0.5)},anchor=west,legend cell align=left,align=left,draw=white!15!black},
]
\addplot [color=mycolor1,dotted,line width=1.5pt,mark=triangle,mark options={solid,rotate=270}]
  table[row sep=crcr]{%
128 2077.28255397045\\
256 8666.01894739673\\
512 34364.5044678294\\
1024 143199.629495367\\
2048 575922.662297887\\
4096 2243781.42842645\\
};
\addlegendentry{$\mu=4$};

\addplot [color=mycolor2,dashed,line width=1.5pt,mark=triangle,mark options={solid,rotate=270}]
  table[row sep=crcr]{%
128 46.4553530035876\\
256 95.2813488833193\\
512 189.915591314194\\
1024 396.680326382795\\
2048 795.06180037716\\
4096 1549.27242164906\\
};
\addlegendentry{$\mu=3$};

\addplot [color=mycolor3,solid,line width=1.5pt,mark=triangle,mark options={solid,rotate=270}]
  table[row sep=crcr]{%
128 2.41116038427345\\
256 2.56510581761287\\
512 2.6689585453212\\
1024 2.94972640686121\\
2048 3.03452946223297\\
4096 3.14360502574124\\
};
\addlegendentry{$\mu=2$};

\addplot [color=mycolor4,dashed,line width=1.5pt,mark=triangle,mark options={solid,rotate=270}]
  table[row sep=crcr]{%
128 1.1746920716352\\
256 1.16289343437043\\
512 1.15263442836347\\
1024 1.14550222790247\\
2048 1.13423050202662\\
4096 1.15441859661499\\
};
\addlegendentry{$\mu=1$};

\addplot [color=mycolor5,dotted,line width=1.5pt,mark=triangle,mark options={solid,rotate=270}]
  table[row sep=crcr]{%
128 1\\
256 1\\
512 1\\
1024 1\\
2048 1\\
4096 1\\
};
\addlegendentry{$\mu=0$};

\end{axis}
\end{tikzpicture}%
}
\caption{Illustration of discrete regularity: the discrete
    $\Hper^{\mu}$-Sobolev seminorm $\norm{\bfu_0}_{\mu,N}$ is
    plotted against the number $N$ of Fourier modes, where $\bfu_0$ is
    chosen as in  \eqref{eq:u0-rand-reg} (and thus corresponds to
    a function in $\Hper^{\alpha}$).
  }
  \label{fig:regularity}
\end{figure}

After these introductory explanations, we now
fix the spatial discretization and set $N=2048$. We consider
\eqref{eq:linschroedi} for two different functions $f$:
\begin{subequations}
\begin{align}
  f(x) & = \sin x,    \label{eq:pot1}\\
  f(x) & =  (x/\pi)^2.    \label{eq:pot2}
\end{align}
\end{subequations}

\begin{example}
In Figure~\ref{fig:error-pot1} we show the numerically observed orders of
the exponential Euler and the Lawson Euler method for the smooth, periodic
potential \eqref{eq:pot1} for different
values of $\alpha$ such that the corresponding initial function
is contained in $\Hper^{\alpha}$.
The leading error terms of the new analysis
for the exponential Euler and the Lawson-Euler method are given
in \eqref{eq:euler-cond} and \eqref{eq:lawson-o1-linear},
respectively. For comparison, we also added the leading error term
\eqref{eq:euler-acta} from our previous work.

Since $B:\Hper^\alpha \to \Hper^\alpha$ is a bounded  perturbation of
$A$, the exact solution of the continuous problem is guaranteed to stay in
$\Hper^\alpha$ for initial values in $\Hper^\alpha$ for $\alpha \geq 0$.
For the discrete problem, $\ee^{-\sigma \sw A_N}$
and $\ee^{\sigma \sw (-A_n+B_N)}$ are unitary matrices,
which means that they leave all discrete Sobolev norms $\norm{\cdot}_{\alpha,N}$
invariant. Thus the expression in  \eqref{eq:lawson-o1-linear} can be bounded
by
\begin{align*}
  \norm{ \ee^{-(1-\sigma)\sw A_N } \comm{A_N,B_N} \ee^{-\sigma \sw A_N} \bfu(t)}_{0,N}
  \leq c_1 \norm{\ee^{-\sigma \sw A_N} \bfu(t)}_{1,N}
  = c_1 \norm{\bfu(t)}_{1,N}
  = c_1 \norm{\bfu_0}_{1,N}.
\end{align*}
Here, the first inequality was proved in \cite[Lemma~3.1]{JahL00} with
a constant $c_1$ independent of $N$ and $A_N$.

Hence, the (sufficient but not necessary) order condition
\eqref{eq:lawson-o1-linear} for the Lawson Euler method yields order one convergence for initial values
bounded in $\norm{\cdot}_\alpha$ for $\alpha \geq 1$.
Numerically, we observe an order reduction
for $\alpha=0$ for the Lawson Euler method, while
the exponential Euler method, which requires initial values in
$\Hper^2=D(A)$,  cf.\ \eqref{eq:euler-cond} or \eqref{eq:euler-acta},
shows
order reduction for $\alpha \leq 1$.
For $\alpha=0$ the error of the exponential Euler method has an
irregular behaviour for larger step sizes. To better visualise the
order, we added thin lines (blue in the colored version) to all curves
related to
$\alpha=0$. The slopes $p$
of these lines are also given in the legends (blue in the colored version).
\end{example}

\begin{figure}[tb]
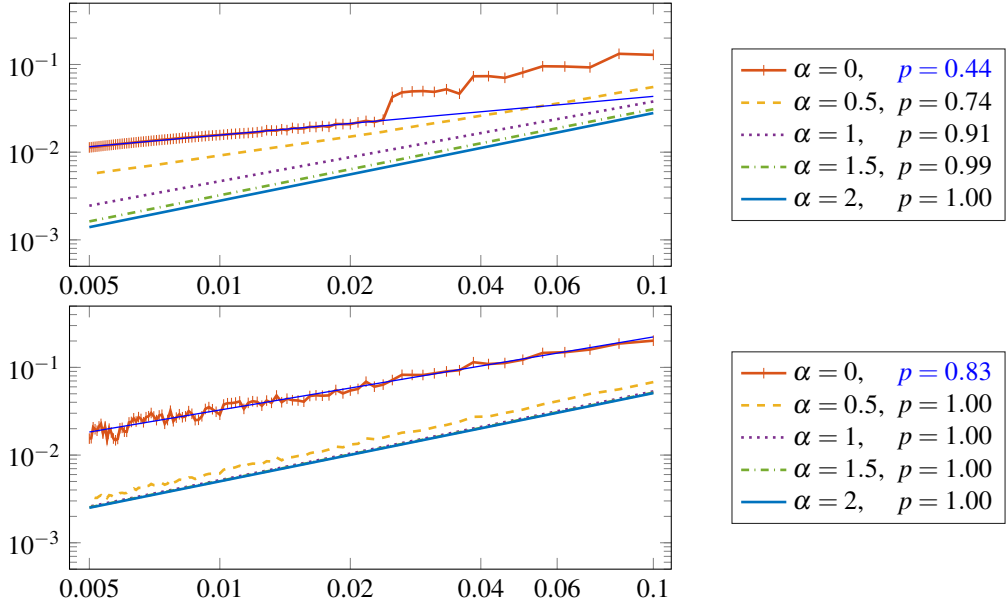

  \centering
  \input{tikz-new/expon_EulerN=2048pot=1.tikz}
  \input{tikz-new/LawsonN=2048pot=1.tikz}
  \caption{Discrete $L^\infty((0,1),L^2(\Omega))$ error of the numerical solution of
    \eqref{eq:linschroedi} with periodic
    potential \eqref{eq:pot1} for the    exponential Euler method (top)
    and the Lawson Euler method
    (bottom) for starting values in $\Hper^\alpha$. The values of $p$ in
    the legend show the numerically observed orders of the schemes.}
  \label{fig:error-pot1}
\end{figure}

\begin{figure}[tb]
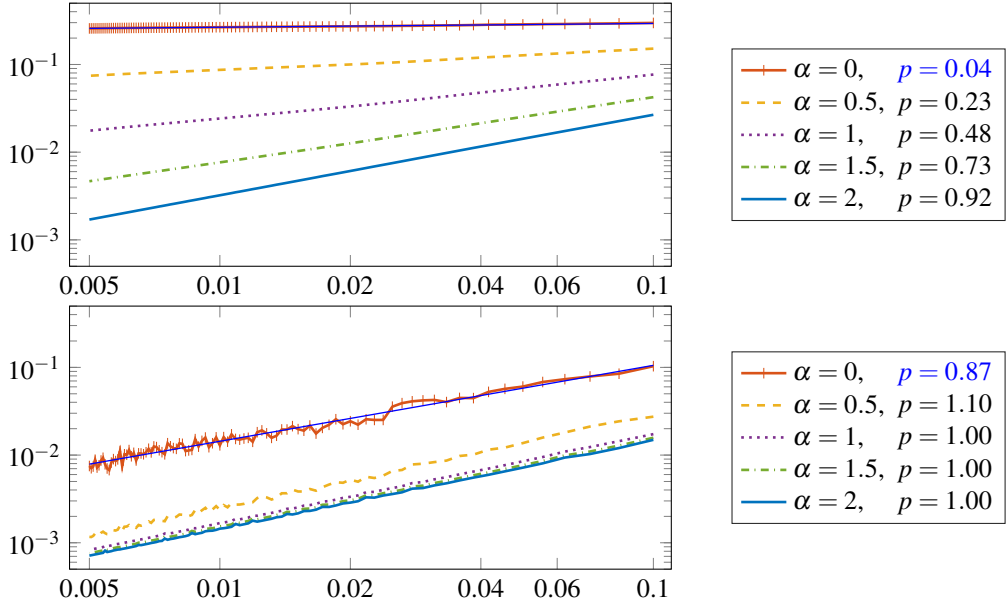

  \centering
  \input{tikz-new/expon_EulerN=2048pot=2.tikz}
  \input{tikz-new/LawsonN=2048pot=2.tikz}
  \caption{Discrete $L^\infty((0,1),L^2(\Omega))$ error of the numerical solution of
    \eqref{eq:linschroedi} with quadratic
    potential \eqref{eq:pot2} for the    exponential Euler method (top)
    and the Lawson Euler method
    (bottom) for starting values in $\Hper^\alpha$. The values of $p$ in
    the legend show the numerically observed orders of the schemes.}
  \label{fig:error-pot2}
\end{figure}

\begin{example}
  In Figure~\ref{fig:error-pot2} we present the same experiment for
  the quadratic potential \eqref{eq:pot2}.  Here, the commutator bound
  of \cite[Lemma~3.1]{JahL00} does not apply, since it requires a
  $C^5$ smooth and periodic potential $f$.  The situations differs
  considerably for the exponential Euler method which suffers
  from order reduction  for all $\alpha\leq 2$ due to the nonsmooth potential $f$.
  In contrast,  the Lawson
  Euler method still converges with order one for $\alpha \geq 0.5$.
\end{example}

Note that for these examples, the convergence behavior is slightly
better than predicted by our theory. This is not a contradiction,
because the order conditions are only sufficient but not necessary. To
be more precise, our analysis contains a worst case estimation of the
error propagation from the local to the global error by using Lady
Windermere's fan in the proof of Theorem~\ref{thm:lawson-error}.
Nevertheless, the examples clearly show the different behavior of the
exponential Euler method and the Lawson-Euler method. Which of the
two methods yields better results depends on the given problem,
as reflected by our error analysis.

\section*{Acknowledgements}

We thank David Hipp and Jan Leibold for helpful discussions and their careful reading of this manuscript and all students of the ``Exponential Integrators'' class at KIT for their inspiration.

We gratefully acknowledge financial support by the Deutsche Forschungsgemeinschaft (DFG) through CRC 1173.


\end{document}